\documentclass[a4paper, 13pt]{article}
\usepackage[english]{babel}

\usepackage[T1]{fontenc}
\usepackage{lmodern}

\usepackage{amsmath,amsthm,amssymb,amsfonts}
\usepackage{mathtools}
\usepackage{thmtools}
\usepackage{mathabx}

\usepackage{bbm}

\usepackage{hyperref}
\usepackage{cleveref}

\usepackage[a4paper,	bindingoffset=0in,		left=1.2in,		right=1.2in,	top=1in,	bottom=1in,	footskip=.25in]{geometry} 

\usepackage[backend=biber,style=alphabetic,sorting=ynt,sortcites]{biblatex}
\allowdisplaybreaks[1]

\theoremstyle{plain}
\newtheorem{theorem}{Theorem}[section]
\newtheorem*{theorem*}{Theorem}

\newtheorem{prop}[theorem]{Proposition}
\theoremstyle{definition}

\newtheorem{dfn}[theorem]{Definition}
\newtheorem{rem}[theorem]{Remark}

\theoremstyle{remark}

\makeatletter
\providecommand*{\twoheadrightarrowfill@}{%
  \arrowfill@\relbar\relbar\twoheadrightarrow
}
\providecommand*{\twoheadleftarrowfill@}{%
  \arrowfill@\twoheadleftarrow\relbar\relbar
}
\providecommand*{\xtwoheadrightarrow}[2][]{%
  \ext@arrow 0579\twoheadrightarrowfill@{#1}{#2}%
}
\providecommand*{\xtwoheadleftarrow}[2][]{%
  \ext@arrow 5097\twoheadleftarrowfill@{#1}{#2}%
}
\newcommand\setItemnumber[1]{\setcounter{enum\romannumeral\@enumdepth}{\numexpr#1-1\relax}}
\makeatother

\DeclareMathOperator\supp{supp}

\newcommand{\R}{\mathbb{R}}

\newcommand{\C}{\mathbb{C}}
\newcommand{\Z}{\mathbb{Z}}

\newcommand{\Rpt}{\mathbb{R}_+^\times}

\AtEveryBibitem{\clearfield{url}}
\usepackage{bbold}
\begin{document}
\title{A groupoid approach to the Wodzicki residue and the Kontsevich-Vishik trace}
\author{Omar Mohsen\footnote{Paris-Saclay University, Paris, France, \texttt{omar.mohsen@universite-paris-saclay.fr}}}
\date{}
\maketitle
\begin{abstract}
    Building on the work of Debord and Skandalis, van Erp and Yuncken introduced a groupoid approach to pseudo-differential operators which has various advantages over the classical approach using Hörmander's symbolic calculus. 
    In a recent work by Couchet and Yuncken, they showed that for pseudo-differential operators of order $-\dim(M)$ acting on a smooth manifold $M$ the Wodzicki residue can be naturally obtained from van Erp and Yuncken's approach. 
    In this short note, we extend their work to pseudo-differential operators of any order. We also describe of the Kontsevich-Vishik trace.
\end{abstract}
\setcounter{tocdepth}{2} 
\tableofcontents
\section*{Introduction}
In the 80s, Connes \cite{ConnesBook} constructed the tangent groupoid and gave a short proof of Atiyah-Singer index theorem. In \cite{DebordSkandalis1}, 
Debord and Skandalis noticed that the tangent groupoid admits a natural $\Rpt$ action which has the crucial property that the crossed product of a natural ideal of the $C^*$-algebra of the 
tangent groupoid is Morita equivalent to the $C^*$-algebra of classical pseudo-differential operators of order $0$. This essentially says that up to Morita equivalent classical pseudo-differential operators are distributions on the tangent groupoid.
This idea was taken further by van Erp and Yuncken who gave an elegant and simple criterion for distributions on the tangent groupoid to come from pseudo-differential operators. 
Furthermore, they showed that one can relatively easily prove the standard properties of a classical pseudo-differential calculus immediately from this criterion.
This approach also has the advantage of being quite easy to generalise to Lie groupoids, for example the $b$-groupoid \cite{MonthubertbGroupoid} and filtered manifolds \cite{DaveHallerHeat}.

Some notions like the principal symbol are quite direct to obtain from van Erp and Yuncken's approach. On the other hand, this is less clear with the Wodzicki residue \cite{WodzickiResidue,GuilleminWodzickiResidue} and the Kontsevich-Vishik trace \cite{KontsevichVishik}.
 Couchet and Yuncken \cite{couchet2023groupoid} recently gave a simple description of the Wodzicki residue for operators of order $-\dim(M)$. In this article, we extend their formula so that it holds for pseudo-differential operators of any order.
 We also derive a closely related formula which we show agrees with the Kontsevich-Vishik trace.

 This approach to the Wodzicki residue and the Kontsevich-Vishik trace generalises directly to the case of filtered manifolds, which has been considered by Dave and Haller \cite{DaveHallerHeat} and Ponge \cite{PongeWodiz1,PongeWodiz2}. 
 We finally remark that the approach given here together with the groupoid constructed by the author \cite{MohsenTangentCone}, 
 might shed some light on how to construct a Wodzicki residue and a Kontsevich-Vishik trace for the more general case of filtered manifolds, where one is given vector fields 
 satisfying Hörmander's condition but no longer supposes that their rank is locally constant. 
\paragraph*{Acknowledgments}
We thank Couchet, Higson and Yuncken for their remarks. A very similar formula to the one given here was found by Higson, Sukochev and Zanin in an unpublished work. We thank the referee for his remarks.
\section{The groupoid approach to classical pseudo-differential calculus}\label{sec:}


Let $M$ be a smooth manifold. We denote by
    \begin{equation*}\begin{aligned}
        \mathbb{T}M=M\times M\times \Rpt \cup TM\times \{0\}
    \end{aligned}\end{equation*}
the tangent groupoid defined by Connes \cite{ConnesBook}. The space $\mathbb{T}M$ is naturally a Lie groupoid. In fact, there is a unique smooth structure on $\mathbb{T}M$ such that the following maps are smooth: \begin{itemize}
    \item    The map 
        \begin{equation*}\begin{aligned}
            \mathbb{T}M\to M\times M\times \Rpt,\quad \pi(y,x,t)\mapsto(y,x,t),\quad \pi(x,X,0)\mapsto(x,x,0),
        \end{aligned}\end{equation*}
        where $x,y\in M$, $X\in T_xM$, $t\in \Rpt$.
        \item The map    \begin{equation*}\begin{aligned}
            \mathbb{T}M\to \R,\quad(y,x,t)\mapsto\frac{f(y)-f(x)}{t},\quad (x,X,0)\mapsto df_x(X),
        \end{aligned}\end{equation*}
        where $f:M\to \R$ is any smooth function.
\end{itemize}

The space $\mathbb{T}M$ is a groupoid whose space of objects is $M\times \R_+$. Its source and range map are given by
    \begin{equation*}\begin{aligned}
        s(y,x,t)=(x,t),\quad s(x,X,0)=(x,0),\\
        r(y,x,t)=(y,t),\quad r(x,X,0)=(x,0),
    \end{aligned}\end{equation*} 
    The inverse and the product are defined by 
        \begin{equation*}\begin{aligned}
            (y,x,t)^{-1}=(x,y,t),\quad (x,X,t)^{-1}=(x,-X,0)\\
            (z,y,t)\cdot (y,x,t)=(z,x,t),\quad (x,X,0)\cdot (x,Y,0)=(x,X+Y,0).
        \end{aligned}\end{equation*}

We denote by $$\Omega^{\frac{1}{2}}:=|\Lambda|^{\frac{1}{2}}\ker(ds)\otimes |\Lambda|^{\frac{1}{2}}\ker(dr).$$
The space $C^\infty_c(\mathbb{T}M,\Omega^{\frac{1}{2}})$ of compactly supported smooth sections of $\Omega^{\frac{1}{2}}$ is naturally a $*$-algebra.
Here, if $E$ is a vector bundle and $\alpha\in \C$, then  $|\Lambda|^\alpha E$ denotes the bundle of $\alpha$-densities on $E$.
We refer the reader to the article of Debord and Lescure \cite{DebordLescureIndextheoreAndGroupoids} for an introduction to Lie groupoids and their convolution algebras, as well as the tangent groupoid. 

The manifold $\mathbb{T}M$ is naturally equipped with a smooth $\Rpt$-action called the Debord-Skandalis action \cite{DebordSkandalis1} given by 
    \begin{equation*}\begin{aligned}
        \alpha_\lambda:\mathbb{T}M\to \mathbb{T}M,\quad \alpha_\lambda(y,x,t)=(y,x,\lambda^{-1}t),\quad  \alpha_\lambda(x,X,0)=(x,\lambda X,0),\quad \lambda\in \Rpt.
    \end{aligned}\end{equation*}
One easily sees that for any $\lambda\in \Rpt$, $\alpha_\lambda$ is a groupoid automorphism. Therefore, $\alpha_\lambda$ acts by algebra automorphism on $C^\infty_c(\mathbb{T}M,\Omega^{\frac{1}{2}})$. 
It also acts on  $C^{-\infty}(\mathbb{T}M,\Omega^{\frac{1}{2}})$, which denotes the topological dual of $C^\infty_c(\mathbb{T}M,(\Omega^{\frac{1}{2}})^*\otimes |\Lambda|^1T(\mathbb{T}M))$, i.e., the space of distributions on $\mathbb{T}M$ with coefficients in $\Omega^\frac{1}{2}$.
We remark that $(\Omega^{\frac{1}{2}})^*\otimes |\Lambda|^1T(\mathbb{T}M)$ is naturally isomorphic to $\Omega^\frac{1}{2}$.
We refer the reader to \cite{LescureManchonVassout} for an introduction on vector bundle valued distributions.

\begin{dfn}[{van Erp and Yuncken \cite{ErikBobCalculus}}]\label{dfn:}Let $k\in \C$. The space $\mathbb{\Psi}^k(\mathbb{T}M,\Omega^\frac{1}{2})$ denotes the subspace of $u\in C^{-\infty}(\mathbb{T}M,\Omega^{\frac{1}{2}})$ which satisfy the following: \begin{enumerate}
    \item   We have 
        \begin{equation}\label{eqn:WF_cond}\begin{aligned}
            \mathrm{WF}(u)\cap \ker(ds)^{\perp}=\mathrm{WF}(u)\cap\ker(dr)^{\perp}=\emptyset.
        \end{aligned}\end{equation}
    \item The maps $r,s:\supp(u)\to M\times \R_+$ are proper.
    \item For any $\lambda\in \R_+$, 
    \begin{equation}\begin{aligned}
        \Phi_u(\lambda):=\lambda^{-k}\alpha_\lambda(u)-u\in  C^{\infty}_c(\mathbb{T}M,\Omega^{\frac{1}{2}})
     \end{aligned}\end{equation}
\end{enumerate}
\end{dfn}
Distributions satisfying \eqref{eqn:WF_cond} are called admissible in \cite{LescureManchonVassout}. They are precisely the distributions for which convolution law coming from the Lie groupoid structure is well-defined, i.e.,
 if $u$ and $v$ are admissible, then one can define $u\ast v$ which is again admissible.

Let $\pi:\mathbb{T}M\to \R_+$ be the natural projection. The distribution $u$ disintegrates along the fibers of $\pi$, i.e., we can define the pullback of $u$
on each fiber of $\pi$, see \cite[Theorem 8.2.4]{HormanderBook1}. We denote by $u_t$ the restriction of $u$ to the fiber of $\pi$ at $t\in \R_+$. Therefore, $u_t\in C^{-\infty}(M\times M,|\Lambda|^\frac{1}{2}T(M\times M))$ for $t>0$ and 
$u_0\in C^{-\infty}(TM,|\Lambda|^1TM)$. 
\begin{theorem}[{van Erp and Yuncken \cite{ErikBobCalculus}}]\label{thm:}
    If $k\in \C$, $u\in \mathbb{\Psi}^k(\mathbb{T}M,\Omega^\frac{1}{2})$, then $u_1$ is the kernel of a properly supported classical pseudo-differential operator of order $k$. Conversely, the kernel of any
     properly supported classical pseudo-differential operator of order $k$ can be written as $u_1$ for some $u\in \mathbb{\Psi}^k(\mathbb{T}M,\Omega^\frac{1}{2})$.
\end{theorem}
The elegance of \Cref{thm:} is that the principal cosymbol of $u_1$ is $u_0$, thus $u$ gives a deformation between the pseudo-differential operator $u_1$ and its cosymbol $u_0$, see \cite[Section 6]{ErikBobCalculus}.

\begin{rem}\label{rem:holomorphic_family}
    The existence $u\in \mathbb{\Psi}^k(\mathbb{T}M,\Omega^\frac{1}{2})$ starting from a properly supported classical pseudo-differential operator is constructive, see \cite[Section 11]{ErikBobCalculus}. 
    In particular if one has a family of properly supported classical pseudo-differential operators which depends holomorphically on some parameter, 
    then one can construct a corresponding family of elements in $\mathbb{\Psi}^k(\mathbb{T}M,\Omega^\frac{1}{2})$ which also depends holomorphically on the same parameter.
\end{rem}
The following proposition, which is a simple consequence of the closed graph theorem, shows that $\Phi_u$ is a jointly smooth function in both $\Rpt$ and $\mathbb{T}M$.
\begin{prop}[{\cite[Lemma 21]{ErikBobCalculus}}]\label{prop:van_erp_smooth}
    Let $k\in   \C$, $u\in \mathbb{\Psi}^k(\mathbb{T}M,\Omega^\frac{1}{2})$. The following map is smooth 
        \begin{equation*}\begin{aligned}
            \Rpt\to C^{\infty}_c(\mathbb{T}M,\Omega^{\frac{1}{2}}),\quad \lambda\to \Phi_u(\lambda).
        \end{aligned}\end{equation*}
\end{prop}

\section{The Wodzicki residue and the Kontsevich-Vishik trace}\label{sec:Wodzicki}

Let $k\in   \C$, $u\in \mathbb{\Psi}^k(\mathbb{T}M,\Omega^\frac{1}{2})$. For any $\lambda\in \Rpt$, since $\Phi_u(\lambda)\in C^{\infty}_c(\mathbb{T}M,\Omega^{\frac{1}{2}})$, we can take its restriction 
on the space of objects, $M\times \R_+$, to obtain a smooth section of $|\Lambda|^1TM$ over $M\times \R_+$. We denote this section by $\Phi_u(\lambda,x,t)$. It satisfies the cocycle relation 
    \begin{equation*}\begin{aligned}
        \Phi_u(\lambda_1\lambda_2,x,t)=\lambda_2^{-k-\dim(M)}\Phi_u(\lambda_1,x,t\lambda_2)+\Phi_u(\lambda_2,x,t).
    \end{aligned}\end{equation*}
This cocycle relation follows immediately from the relation 
    \begin{equation*}\begin{aligned}
        \Phi_u(\lambda_1\lambda_2)=\lambda_2^{-k}\alpha_{\lambda_2}(\Phi_u(\lambda_1))+\Phi_u(\lambda_2).
    \end{aligned}\end{equation*}
    The factor $\lambda_2^{-\dim(M)}$ comes the action of $\Rpt$ on densities, see for example \cite[Proposition 3.2]{LefschetzGroupoid} for more details.

    \begin{prop}\label{prop:main_lemma}
        Let $V$ be a finite dimensional complex vector space, $k\in \C$, $\phi:\Rpt\times \R_+\to V$ a smooth function which satisfies the relation 
            \begin{align}
                \phi(\lambda_1\lambda_2,t)=\lambda_2^{-k}\phi(\lambda_1,t\lambda_2)+\phi(\lambda_2,t),\quad \forall t\in \R_+,\lambda_1,\lambda_2\in \Rpt.\tag{$C_k$}
            \end{align}
         Then, there exists a smooth map $\psi:\R_+\to V$ and $c\in V$ such that 
                \begin{equation}\label{eqn:form2}\begin{aligned}
                    \phi(\lambda,t)=\lambda^{-k}\psi(\lambda t)-\psi(t)+ t^k\log(\lambda)c,\quad \forall (\lambda,t)\in \Rpt\times \R_+
                \end{aligned}\end{equation}
                Furthermore: \begin{itemize}
                    \item  If $k\notin \Z_{\geq 0}$, then $c=0$ and $\psi$ is unique.
                    \item If $k\in \Z_{\geq 0}$, then $c$ is unique and is given by 
                        \begin{equation}\label{eqn:qsdfqsdf}\begin{aligned}
                            c=\frac{1}{k!\log(\lambda)}\frac{d^k}{dt^k}\Bigr|_{t=0}\phi(\lambda,t),
                        \end{aligned}\end{equation}
                     which we show is independent of $\lambda$.
                \end{itemize}
    \end{prop}
    \begin{rem}\label{rem:after_prop}
     Before we give the proof let us explain the motivation behind \Cref{prop:main_lemma}. 
    If we apply \Cref{prop:main_lemma} to $\Phi_u(\cdot,x,\cdot)$ with $u\in C^\infty_c(\mathbb{T}M,\Omega^\frac{1}{2})$, then there is an obvious choice for $\psi$ which is simply the restriction of $u$ itself on the space of objects, i.e., the diagonal.
     Therefore, \Cref{prop:main_lemma} gives a way to extend the definition of the restriction to the diagonal beyond the case $u\in  C^\infty_c(\mathbb{T}M,\Omega^\frac{1}{2})$. This is precisely the Kontsevich-Vishik trace.
    \end{rem}
    
     \begin{proof}
        The proof is divided into a series of straightforward to check steps.
        \begin{enumerate}
            \item If $\psi:\R_+\to V$ is any smooth function, then $(\lambda,t)\mapsto\lambda^{-k}\psi(\lambda t)-\psi(t)$ satisfies $(C_k)$.
            \item Suppose $\Re(k)<0$. If $\phi:\Rpt\times \R_+\to V$ is any smooth function that satisfies $(C_k)$, and $\phi(\frac{1}{2},t)=0$ for all $t\geq 0$, then $\phi=0$. To see this, we apply $(C_k)$ twice to obtain 
                \begin{equation*}\begin{aligned}
                    \phi\left(\frac{\lambda}{2},t\right)=\lambda^{-k}\phi\left(\frac{1}{2},t\lambda\right)+\phi(\lambda,t)=\phi(\lambda,t)\\
                    \phi\left(\frac{\lambda}{2},t\right)=2^k\phi\left(\lambda,\frac{t}{2}\right)+\phi\left(\frac{1}{2},t\right)=2^k\phi\left(\lambda,\frac{t}{2}\right).
                \end{aligned}\end{equation*}
            Therefore, $$\phi(\lambda,t)=2^k\phi\left(\lambda,\frac{t}{2}\right)$$ which by recurrence implies that $$\phi(\lambda,t)=2^{kn}\phi\left(\lambda,\frac{t}{2^n}\right).$$ Taking the limit as $n\to +\infty$, we conclude.
            \item    Suppose $\Re(k)<0$ and $\phi:\Rpt\times \R_+\to V$ a smooth function that satisfies $(C_k)$. We define 
                \begin{equation}\label{eqn:psirezneg}\begin{aligned}
                    \psi(t):=-\sum_{j=0}^\infty \phi\left(\frac{1}{2},\frac{t}{2^j}\right)2^{jk}
                \end{aligned}\end{equation}
            It is straightforward to see that $\psi$ is smooth and that $$\phi\left(\frac{1}{2},t\right)=2^k\psi\left(\frac{t}{2}\right)-\psi\left(t\right).$$ Therefore, the function
             $$(\lambda,t)\mapsto \phi(\lambda,t)-\lambda^{-k}\psi(\lambda t)-\psi(t)$$ satisfies $(C_k)$ and vanishes when $\lambda=\frac{1}{2}$. Hence, by Step 2, 
                 \begin{equation*}\begin{aligned}
                     \phi(\lambda,t)=\lambda^{-k}\psi(\lambda t)-\psi(t).
                 \end{aligned}\end{equation*}
             This finishes the proof of the existence of $\psi$ in the case $\Re(k)<0$.
             \item If $\phi$ satisfies $(C_k)$ and $\phi(\lambda,0)=0$ for all $\lambda\in \Rpt$, then $t^{-1}\phi$ defines a smooth function which satisfies $(C_{k-1})$. Furthermore, if $t^{-1}\phi$ has the form \eqref{eqn:form2},
             then $\phi$ also has the same form. 
             \item If $\phi$ satisfies $(C_0)$, then it is clear that $\phi(\lambda,0)=c\log(\lambda)$ for some $c\in V$.
             \item If $\phi$ satisfies $(C_k)$ and $k\neq 0$, then there exists a unique $v\in V$ such that $\phi(\lambda,0)=(\lambda^{-k}-1)v$. To see this, first suppose $\Re(k)\neq 0$, then by applying $C_k$ twice, we have
                  \begin{equation*}\begin{aligned}
                      \phi(\lambda_1\lambda_2,0)=\lambda_2^{-k}\phi(\lambda_1,0)+\phi(\lambda_2,0)\\
                      \phi(\lambda_1\lambda_2,0)=\lambda_1^{-k}\phi(\lambda_2,0)+\phi(\lambda_1,0)
                  \end{aligned}\end{equation*}
                By taking the difference, we deduce that 
                    \begin{equation*}\begin{aligned}
                        \frac{\phi(\lambda_1,0)}{\lambda_1^{-k}-1}=  \frac{\phi(\lambda_2,0)}{\lambda_2^{-k}-1}.
                    \end{aligned}\end{equation*}
                The result follows. If $\Re(k)=0$, then there are only discrete values of $\lambda$ for which $\lambda^{-k}=1$. We apply the above argument for $\lambda_1,\lambda_2$ avoiding such values together with continuity of $\phi$ to conclude.
                \item The proof of the existence of $\psi$ and $c$ is now a straightforward recursive argument combining Steps $4,5,6$ together with the case $\Re(k)<0$ which we already proved.
                \item Suppose $k\not\in \Z_{\geq 0}$. We will show that $c=0$ and $\psi$ is unique. We have $c=0$ because \eqref{eqn:form2} implies that $t^k\log(\lambda)c$ is smooth at $t=0$.
                By taking the difference of two $\psi$'s which satisfy \eqref{eqn:form2}, without loss of generality, we can suppose that $$\lambda^{-k}\psi(\lambda t)-\psi(t)=0.$$
                This implies that $$\psi(t)=t^k\psi(1)$$
                The left-hand side is smooth at $t=0$ while the right-hand side is smooth only if $\psi(1)=0$. This shows that $\psi=0$.
                \item If $k\in \Z_{\geq 0}$, we will show that $c$ is unique. By taking the difference of two $\psi$ and $c$ which satisfy \eqref{eqn:form2}, without loss of generality, we can suppose that
                $$\lambda^{-k}\psi(\lambda t)-\psi(t)=-t^k\log(\lambda)c.$$ This implies that $$\psi(t)=t^k\psi(1)-t^k\log(t)c.$$ This is smooth at $t=0$ only if $c=0$.
                \item It remains to show that if $k\in \Z_{\geq 0}$, then $c$ is given by \eqref{eqn:qsdfqsdf}. This follows by noticing that if $\phi$ satisfies $(C_k)$, then $\frac{d}{dt}\phi$ satisfies $(C_{k-1})$. We then apply Step 5.\qedhere
        \end{enumerate}
    \end{proof}
    We now apply \Cref{prop:main_lemma} to $\Phi_u(\lambda,x,t)$.
    \begin{enumerate}
        \item If $\Re(k)\not\in  \Z_{\geq -\dim(M)}$, we obtain a smooth section $\Psi_u(x,t)$ of $|\Lambda|^1TM$ over $M\times \R_+$ such that 
            \begin{equation*}\begin{aligned}
                \Phi_u(\lambda,x,t)=\lambda^{-k-\dim(M)}\Psi_u(x,\lambda t)-\Psi_u(x,t).
            \end{aligned}\end{equation*}
        \item If $\Re(k)\in  \Z_{\geq -\dim(M)}$, we obtain a smooth section of $|\Lambda|^1TM$ over $M$ given by
            \begin{equation*}\begin{aligned}
                C_u(x):=\frac{1}{l!\log(\lambda)}\frac{d^{l}}{dt^{l}}\Bigr|_{t=0}\Phi_u(\lambda,x,t),
            \end{aligned}\end{equation*}
            where $l=k+\dim(M)$.
    \end{enumerate}
    \begin{theorem}\label{thm:main_thm}
        Let $u\in \mathbb{\Psi}^k(\mathbb{T}M,\Omega^\frac{1}{2})$. \begin{enumerate}
            \item    If $\Re(k)\not\in \Z_{\geq -\dim(M)}$, then the Kontsevich-Vishik density of $u_1$ at $x\in M$ is equal to $\Psi_u(x,1)$. In particular, if $M$ is compact, then
                \begin{equation*}\begin{aligned}
                \mathrm{TR}(u_1)=\int_M     \Psi_u(x,1)
                \end{aligned}\end{equation*}
            \item If $\Re(k)\in  \Z_{\geq -\dim(M)}$, then the Wodzicki density of $u_1$ at $x\in M$ is equal to $C_u(x)$. In particular, if $M$ is compact, then
            \begin{equation*}\begin{aligned}
            \mathrm{Res}(u_1)=\int_M     C_u(x)
            \end{aligned}\end{equation*}
        \end{enumerate}
    \end{theorem}
    \begin{proof}
        If $u\in C^\infty_c(\mathbb{T}M,\Omega^{\frac{1}{2}})$, then by \Cref{rem:after_prop}, we know  that $\Psi_u(x,1)$ agrees with the restriction to the diagonal of $u_1$, which is by definition the Kontsevich-Vishik density of $u_1$ at $x\in M$.
        Since the Kontsevich-Vishik trace is a holomorphic extension of the trace for trace class operators whose residue is the Wodzicki density, see for example \cite[Théorème 2.3]{VassoutThesis}. \Cref{thm:main_thm} follows from \Cref{rem:holomorphic_family} and the following Proposition.\qedhere
        
        \begin{prop}\label{prop:}
        Let $V$ be a finite dimensional complex vector space, $U\subseteq \C$ an open subset, $\phi:U\times \Rpt\times\R_+\to V$ a smooth function which is holomorphic in $z\in U$ such that $\phi(z,\cdot,\cdot)$ satisfies $(C_z)$ for every $z\in U$. 
        \begin{itemize}
            \item             For every $z\in U\backslash \Z_{\geq 0}$, let $\psi(z,t)$ be the unique map obtained from \Cref{prop:main_lemma} applied to $\phi(z,\cdot,\cdot)$. 
        \item For every $z\in U\cap \Z_{\geq 0}$, let $c(z)$ be the element obtained from \eqref{eqn:qsdfqsdf} applied to $\phi(z,\cdot,\cdot)$.
    \end{itemize}
        Then, the map $\psi(z,t)$ is smooth in $t$ and meromorphic in $z$. It has simple poles at $z\in U\cap \Z_{\geq 0}$ whose residue is equal to $-c(z)$.
        \end{prop}
        \begin{proof}
            By \eqref{eqn:psirezneg}, $\psi$ is clearly holomorphic on $\{z\in U:\Re(z)<0\}$. We will show that it extends to $\{z\in U:\Re(z)<1\}$ with a simple pole at $z=0$.
             For $z\neq 0$, let $v(z)$ be the unique element of $V$ such that $\phi(z,\lambda,0)=(\lambda^{-z}-1)v(z)$. It is clear that $v(z)$ is holomorphic at $z\neq 0$. It has a simple pole at $z=0$ which is equal to 
             $$\lim_{z\to 0}\frac{z\phi(z,\lambda,0)}{\lambda^{-z}-1}=-\frac{\phi(0,\lambda,0)}{\log(\lambda)}=-c(0).$$ We write $$\phi(z,\lambda,t)=t\frac{\phi(z,\lambda,t)-\phi(z,\lambda,0)}{t}+\phi(z,\lambda,0).$$ 
             The function $t^{-1}(\phi(z,\lambda,t)-\phi(z,\lambda,0))$ is holomorphic in $z$ and satisfies $(C_{z-1})$.
             Let $\tilde{\psi}(z,t)$ be the map obtained from \eqref{eqn:psirezneg} applied to $t^{-1}(\phi(z,\lambda,t)-\phi(z,\lambda,0))$. The function $\tilde{\psi}(z,t)$ is clearly holomorphic in $\{z\in U:\Re(z)<1\}$.
             By uniqueness of $\psi$ in \Cref{prop:main_lemma} for $k\notin  \Z_{\geq 0}$, we deduce that $$\psi(z,t)=t\tilde{\psi}(z,t)+v(z).$$ Hence, $\psi(z,t)$ is meromorphic on $\{z\in U:\Re(z)<1\}$ with a simple pole at $z=1$ whose residue is $-c(0)$.
              The case of $\{z\in U:\Re(z)<n\}$ is done by induction and is left to the reader.
        \end{proof}
    \end{proof}
\begin{refcontext}[sorting=nyt]
\printbibliography
\end{refcontext}
\end{document}